\documentclass[11pt,twoside]{article}
\usepackage{a4}
\usepackage{amssymb,amsmath,amsthm,latexsym}
\usepackage{amsfonts}
  \usepackage{amsfonts}
\usepackage{graphicx}
\usepackage{hyperref}

\newtheorem{theorem}{Theorem}[section]

\newtheorem{proposition}[theorem]{Proposition}
\newtheorem{remark}[theorem]{Remark}

\usepackage{enumerate}

\vspace{5cm}

\begin{document}
  
  \label{'ubf'}  
\setcounter{page}{1}                                 

\markboth {\hspace*{-9mm} \centerline{\footnotesize \sc
    Upper bounds for the prime divisors of friends of 10  }
                 }
                { \centerline                           {\footnotesize \sc  
         Sourav Mandal and Sagar Mandal                                                 } \hspace*{-9mm}              
               }

\vspace*{-2cm}

\begin{center}
{ 
       {\large \textbf { \sc Upper bounds for the prime divisors of friends of 10
                               }
       }
\\

\medskip

{\sc Sourav Mandal }\\
{\footnotesize Department of Mathematics, Ramakrishna Mission Vivekananda Educational and Research Institute Belur Math, Howrah, West Bengal 711202, India}\\

{\footnotesize e-mail: {\it souravmandal1729@gmail.com}}
}

{\sc Sagar Mandal }\\
{\footnotesize Department of Mathematics and Statistics, Indian Institute of Technology Kanpur\\ Kalyanpur, Kanpur, Uttar Pradesh 208016, India}\\

{\footnotesize e-mail: {\it sagarmandal31415@gmail.com}}

\end{center}

\thispagestyle{empty}

\hrulefill

\begin{abstract}  
{\footnotesize  In this paper we propose necessary upper bounds for the second, third and fourth smallest prime divisors of friends of 10 based on the number of distinct prime divisors of it.
}
 \end{abstract}
 \hrulefill



\section{Introduction}
For any positive integer $n$, $I(n)=\frac{\sigma(n)}{n}$ is called the abundancy index of $n$ where $\sigma(n)$ is sum of  positive divisors of $n$. A positive integer $n$ is said to be friendly number if there exist a positive integer $m$ other than $n$ having same abundancy index as $n$ i.e., $I(n)=I(m)$ and in that case we say $m$ is a friend of $n$. If $n$ has no friend then we say $n$ is a solitary number. $10$ is the smallest positive integer whose classification in terms of solitary and friendly is unknown. J. Ward \cite{ward2008does} proved that if 10 has a friend then it is an odd square having at least six distinct prime divisors with $5$ being the least one and there exists a prime divisor congruent to $1$ modulo $3$. In \cite{SS2024} it has been proven that the friendly number of 10 necessarily has at least seven distinct prime divisors and has two primes $p$ and $q$ not necessarily distinct such that $p\equiv 1 \pmod6$ and $q\equiv 1 \pmod{10}$, also it has prime divisor $r$ such that $2a+1\equiv 0 \pmod{f}$ where $f$ is the least odd positive integer greater than $1$ satisfying $5^f\equiv 1 \pmod r$, provided $5^{2a}\mid\mid n~,~a\in\mathbb{Z^+}$(where $n$ is the friendly number of 10). In this paper, we use $\omega(n)$ to denote the number of distinct prime divisors of $n$ and $\lceil . \rceil$ to denote the ceiling function, and we prove the following,\\
\begin{theorem}\label{thm 1.1}
Let $n$ be a friend of 10 and $q_2$ be the second  smallest prime divisor of $n$. Then necessarily \\
\begin{align*}
    7\leq q_2 < {\lceil\frac{7\omega(n)}{3}\rceil}\biggl\{\log\left({\lceil\frac{7\omega(n)}{3}\rceil}\right)+2\log\log\left({\lceil\frac{7\omega(n)}{3}\rceil}\right)\biggl\}.
\end{align*}
\end{theorem}
\begin{theorem}\label{thm 1.2}
Let $n$ be a friend of 10 and $q_3$ be the third smallest prime divisor of $n$. Then necessarily\\ 
\begin{align*}
    11\leq q_3 < {\lceil\frac{180\omega(n)}{41}\rceil}\biggl\{\log\left({\lceil\frac{180\omega(n)}{41}\rceil}\right)+2\log\log\left({\lceil\frac{180\omega(n)}{41}\rceil}\right)\biggl\}.
\end{align*}
\end{theorem}
\begin{theorem}\label{thm 1.3}
Let $n$ be a friend of 10 and $q_4$ be the fourth smallest prime divisor of $n$. Then necessarily \\
\begin{align*}
    13\leq q_4<  {\lceil\frac{390\omega(n)}{47}\rceil}\biggl\{\log\left({\lceil\frac{390\omega(n)}{47}\rceil}\right)+2\log\log\left({\lceil\frac{390\omega(n)}{47}\rceil}\right)\biggl\}.
\end{align*}
\end{theorem}

\section{Preliminaries}
In this section we note some elementary properties of the abundancy index and some useful properties  of the ceiling function.\\\\
\textbf{Properties of Abundancy Index} \cite{rl,paw} \textbf{:}
\begin{enumerate}[(1)]
    \item $I(n)$ is weakly multiplicative that is, if $n$ and $m$ are two coprime positive integers then $I(nm)=I(n)I(m)$.
    \item\label{p2} Let $a,n$ be two positive integers and $a>1$. Then $I(an)>I(n)$.
\item Let $p_1$, $p_2$, $p_3$,..., $p_m$ be $m$ distinct primes and $a_1$, $a_2$, $a_3$ ,..., $a_m$ are positive integers then
\begin{align*}
    I\biggl (\prod_{i=1}^{m}p_i^{a_i}\biggl)=\prod_{i=1}^{m}\biggl(\sum_{j=0}^{a_i}p_i^{-j}\biggl)=\prod_{i=1}^{m}\frac{p_i^{a_i+1}-1}{p_i^{a_i}(p_i-1)}.
\end{align*}
\item\label{p4} If $p_{1}$,...,$p_{m}$ are distinct primes, $q_{1}$,...,$q_{m}$ are distinct primes such that  $p_{i}\leq q_{i}$ for all $1\leq i\leq m$. If $t_1,t_2,...,t_m$ are positive integers then
\begin{align*}
   I \biggl(\prod_{i=1}^{m}p_i^{t_i}\biggl)\geq I\biggl(\prod_{i=1}^{m}q_i^{t_i}\biggl).
\end{align*}
\item\label{p5}  If $n=\prod_{i=1}^{m}p_i^{a_i}$, then $I(n)<\prod_{i=1}^{m}\frac{p_i}{p_i-1}$.
\end{enumerate}
\vspace{5mm}
\textbf{Properties of Ceiling Function $\lceil . \rceil$ :}\\
 For $n\in \mathbb{Z^+}$ and $x\in\mathbb{R^+}$
 \begin{enumerate}[(i)]
     
    \item\label{p6} $\lceil x+n \rceil = \lceil x \rceil + n$ 
    \item\label{p7} $\lceil x \rceil - \lfloor x \rfloor = 1$ if $x\not \in \mathbb{Z^+}$
     \item\label{p8} $\lceil x \rceil - \lfloor x \rfloor = 0$ if $x \in \mathbb{Z^+}$
    \item\label{p9} $\{x\}=x-\lfloor x \rfloor$
   
    \end{enumerate}
    where $\lfloor . \rfloor$ is floor function and $\{.\}$ is fractional part function.
\section{Proof of the main theorems}
\begin{proposition}\label{propo 1}
Let a function $\psi: (1,\infty)	\rightarrow\mathbb{R}$ be defined by $\psi(x)=\frac{x}{x-1}$ then $\psi$ is a strictly decreasing function of $x$ in $(1,\infty)$.
\end{proposition}
\begin{proof}
Let $\psi: (1,\infty)	\rightarrow\mathbb{R}$ be defined by $\psi(x)=\frac{x}{x-1}$. Since $\psi(x)$ is a rational function and $x\neq 1$, $\psi(x)$ is differentiable. Therefore
$$\psi'(x)=\frac{-1}{(x-1)^2}<0~~~~ \text{for all}~x\in (1,\infty).$$
This proves that $\psi$ is a strictly decreasing function of $x$ in $(1,\infty)$.
\end{proof}

\begin{proposition}\label{propo 2}
Let a function $\Omega: [\alpha,\infty)	\rightarrow\mathbb{R}$ be defined by $\Omega(x)=\frac{ax-b}{cx-d}$ ($a,b,c,d,\alpha\in\mathbb{Z^+}$) such that 
 $bc>ad$,$\alpha > \frac{d}{c}$ then $\Omega$ is a strictly increasing of $x$ in $[\alpha,\infty)$.
\end{proposition}
\begin{proof}
Let $\Omega: [\alpha,\infty)	\rightarrow\mathbb{R}$ be defined by $\Omega(x)=\frac{ax-b}{cx-d}$ ($a,b,c,d,\alpha\in\mathbb{Z^+}$) where $bc>ad$, $\alpha > \frac{d}{c}$. Since $\Omega(x)$ is a rational function and $x\neq \frac{d}{c}$, $\Omega(x)$ is differentiable. Therefore
$$\Omega'(x)=\frac{bc-ad}{(cx-d)^2}>0~~~~ \text{for all}~ x\in [\alpha,\infty).$$
This proves that $\Omega$ is a strictly increasing function of $x$ in $[\alpha,\infty)$.
\end{proof}
\begin{remark}\label{r3}
If $p_n$ is the $n\text{-th}$ prime number then $p_n>n$ for each $n\in\mathbb{Z^+}$. Using Proposition \ref{propo 1} we can say that $\frac{p_n}{p_n-1}<\frac{n}{n-1}$ for each $n\in\mathbb{Z^+}$.
\end{remark}
\begin{theorem}\label{thm 3.4}
    If $p_n$ is $n\text{-th}$ prime number then
    $$p_n<n(\log(n)+2\log\log(n))$$
    for $n\geq 4$.
\end{theorem}
 \begin{proof}
  Refer to Theorem 2 in \cite{Ro1} for a proof.
\end{proof}

\subsection{Proof of Theorem \ref{thm 1.1}}
Let $n=5^{2a_1}\cdot \prod_{2\leq i \leq \omega(n)}q^{2a_i}_i$, where $q_i$ are prime numbers greater than 5 and $2a_i$ are positive integer, be a friend of 10. To prove this theorem, it is enough to show that $q_2$ must be strictly less than $p_{\lceil\frac{7\omega(n)}{3}\rceil}$ by Theorem \ref{thm 3.4}. If possible, suppose that $q_2\geq p_{\lceil\frac{7\omega(n)}{3}\rceil}$. Then by Property (\ref{p4}) and Property (\ref{p5}) we have 
\begin{align*}
    I(n)\leq I\biggl(5^{2a_1}\cdot \prod_{2\leq i \leq \omega(n)}p_{\lceil\frac{7\omega(n)}{3}\rceil+i-2}^{2a_i}\biggl)<\frac{5}{4} \cdot \prod_{2\leq i \leq \omega(n)}\frac{p_{\lceil\frac{7\omega(n)}{3}\rceil+i-2}}{p_{\lceil\frac{7\omega(n)}{3}\rceil+i-2}-1}.
\end{align*}
Using Remark \ref{r3} we get 
\begin{align*}
I(n)<\frac{5}{4}\cdot \prod_{2\leq i \leq \omega(n)}\frac{\lceil\frac{7\omega(n)}{3}\rceil+i-2}{\lceil\frac{7\omega(n)}{3}\rceil+i-3}=\frac{5}{4} \cdot \frac{\lceil\frac{7\omega(n)}{3}\rceil+\omega(n)-2}{\lceil\frac{7\omega(n)}{3}\rceil-1}.
\end{align*}
Now we will show that for any $\omega(n)\in\mathbb{Z^+}$ we have  
\begin{align*}
\frac{\lceil\frac{7\omega(n)}{3}\rceil+\omega(n)-2}{\lceil\frac{7\omega(n)}{3}\rceil-1}<\frac{10}{7}.
\end{align*}
Using Property (\ref{p6}) we obtain 
\begin{align*}
\frac{\lceil\frac{7\omega(n)}{3}\rceil+\omega(n)-2}{\lceil\frac{7\omega(n)}{3}\rceil-1}=\frac{\lceil2\omega(n)+\frac{\omega(n)}{3}\rceil+\omega(n)-2}{\lceil2\omega(n)+\frac{\omega(n)}{3}\rceil-1}=\frac{3\omega(n)-2+\lceil\frac{\omega(n)}{3}\rceil}{2\omega(n)-1+\lceil\frac{\omega(n)}{3}\rceil}.
\end{align*}
\vspace{5mm}
We now see the behavior of $\frac{3\omega(n)-2+\lceil\frac{\omega(n)}{3}\rceil}{2\omega(n)-1+\lceil\frac{\omega(n)}{3}\rceil}$ based on the divisibility of $\omega(n)$ by $3$.\\ 
If $3\nmid \omega(n)$ then $\frac{\omega(n)}{3}\not\in \mathbb{Z^+}$ therefore using Property (\ref{p7}) and Property (\ref{p9}) we get
\begin{align*}
    \frac{3\omega(n)-2+\lceil\frac{\omega(n)}{3}\rceil}{2\omega(n)-1+\lceil\frac{\omega(n)}{3}\rceil}&=\frac{3\omega(n)-2+1+\frac{\omega(n)}{3}-\{\frac{\omega(n)}{3}\}}{2\omega(n)-1+1+\frac{\omega(n)}{3}-\{\frac{\omega(n)}{3}\}}\\
    &=\frac{10\omega(n)-3-3\{\frac{\omega(n)}{3}\}}{7\omega(n)-3\{\frac{\omega(n)}{3}\}}.
\end{align*}
Note that for any positive integer  $Q$ which is not divisible by $3$ can be written in the form $Q=3q+r$, where $q\in \mathbb{Z}_{\geq 0}$ and $r\in \{1,2\}$. Therefore, in particular for $Q=\omega(n)$, we have $\{\frac{\omega(n)} {3}\}\in \{\frac{1}{3},\frac{2}{3}\}$ and so
\begin{align*}
    1\leq 3 \{\frac{\omega(n)} {3}\} \leq 2
    \end{align*}
    that is,
\begin{align}\label{e1}
    10\omega(n)-5 \leq 10\omega(n)-3-3\{\frac{\omega(n)} {3}\} \leq 10\omega(n)-4
    \end{align}
and 
    \begin{align}\label{e2}
     7\omega(n)-2 \leq 7\omega(n)-3\{\frac{\omega(n)} {3}\} \leq 7\omega(n)-1.
\end{align}

Using (\ref{e1}) and (\ref{e2}) we finally have
\begin{align*}
    \frac{3\omega(n)-2+\lceil\frac{\omega(n)}{3}\rceil}{2\omega(n)-1+\lceil\frac{\omega(n)}{3}\rceil}\leq \frac{10 \omega(n)-4}{7\omega(n)-2}.
\end{align*}
Now define $f: [1,\infty)	\rightarrow\mathbb{R}$ by $f(t)=\frac{10t-4}{7t-2}$. Then $f$ is a strictly increasing function of t in $[1,\infty)$ by Proposition \ref{propo 2}. Since $\lim_{t\rightarrow \infty} f(t)=\frac{10}{7}$ we have $f(t)<\frac{10}{7}$ for all $t\in [1,\infty)$. In particular for $t=\omega(n)$ we have
\begin{align*}
\frac{10\omega(n)-4}{7\omega(n)-2}<\frac{10}{7} 
\end{align*}
which immediately implies that 
\begin{align*}
 \frac{\lceil\frac{7\omega(n)}{3}\rceil+\omega(n)-2}{\lceil\frac{7\omega(n)}{3}\rceil-1}<\frac{10}{7}.  
\end{align*}

Now if $3\mid \omega(n)$ then $\frac{\omega(n)}{3}\in \mathbb{Z^+}$. Therefore
\begin{align*}
   \frac{3\omega(n)-2+\lceil\frac{\omega(n)}{3}\rceil}{2\omega(n)-1+\lceil\frac{\omega(n)}{3}\rceil}=\frac{3\omega(n)-2+\frac{\omega(n)}{3}}{2\omega(n)-1+\frac{\omega(n)}{3}}=\frac{10\omega(n)-6}{7\omega(n)-3}.
\end{align*}
Now define $g: [1,\infty)	\rightarrow\mathbb{R}$ by $g(t)=\frac{10t-6}{7t-3}$. Then $g$ is a strictly increasing function of t in $[1,\infty)$ by Proposition \ref{propo 2}. Since $\lim_{t\rightarrow \infty} g(t)=\frac{10}{7}$ we have $g(t)<\frac{10}{7}$ for all $t\in [1,\infty)$. In particular for $t=\omega(n)$ we have
\begin{align*}
\frac{10\omega(n)-6}{7\omega(n)-3}<\frac{10}{7} 
\end{align*}
which immediately implies that 
\begin{align*}
 \frac{\lceil\frac{7\omega(n)}{3}\rceil+\omega(n)-2}{\lceil\frac{7\omega(n)}{3}\rceil-1}<\frac{10}{7}.
\end{align*}

Therefore for any $\omega(n)\in \mathbb{Z^+}$ we have
\begin{align*}
 \frac{\lceil\frac{7\omega(n)}{3}\rceil+\omega(n)-2}{\lceil\frac{7\omega(n)}{3}\rceil-1}<\frac{10}{7} 
\end{align*}
which shows that
\begin{align*}
I(n)<\frac{5}{4}\cdot\frac{\lceil\frac{7\omega(n)}{3}\rceil+\omega(n)-2}{\lceil\frac{7\omega(n)}{3}\rceil-1}<\frac{25}{14}<\frac{9}{5}.
\end{align*}
Therefore for $q_2\geq p_{\lceil\frac{7\omega(n)}{3}\rceil}$, $n$ can not be a friend of $10$. Hence necessarily $q_2<p_{\lceil\frac{7\omega(n)}{3}\rceil}$.

\subsection{Proof of Theorem \ref{thm 1.2}}
Let $n=5^{2a_1}\cdot\prod_{2\leq i \leq \omega(n)}q^{2a_i}_i$ be a friend of 10. To prove this theorem it is enough to show that $q_3$ must be strictly less than $p_{\lceil\frac{180\omega(n)}{41}\rceil}$ by Theorem \ref{thm 3.4}. If possible suppose that $q_3\geq p_{\lceil\frac{180\omega(n)}{41}\rceil}$. Then by Property (\ref{p4}) and Property (\ref{p5}) we have 
\begin{align*}
    I(n)\leq I\biggl(5^{2a_1} \cdot 7_{2}^{2a_2}\cdot \prod_{3\leq i \leq \omega(n)}p_{\lceil\frac{180\omega(n)}{41}\rceil+i-3}^{2a_i}\biggl)<\frac{5}{4}\cdot\frac{7}{6}\cdot\prod_{3\leq i \leq \omega(n)}\frac{p_{\lceil\frac{180\omega(n)}{41}\rceil+i-3}}{p_{\lceil\frac{180\omega(n)}{41}\rceil+i-3}-1}.
\end{align*}

Using Remark \ref{r3} we get 
\begin{align*}
     I(n)< \frac{5}{4}\cdot\frac{7}{6}\prod_{3\leq i \leq \omega(n)}\frac{{\lceil\frac{180\omega(n)}{41}\rceil+i-3}}{{\lceil\frac{180\omega(n)}{41}\rceil+i-3}-1}=\frac{5}{4}\cdot\frac{7}{6}\cdot \frac{{\lceil\frac{180\omega(n)}{41}\rceil+\omega(n)-3}}{{\lceil\frac{180\omega(n)}{41}\rceil}-1}.
\end{align*}
Now we will show that for any $\omega(n)\in\mathbb{Z^+}$ we have
\begin{align*}
   \frac{{\lceil\frac{180\omega(n)}{41}\rceil+\omega(n)-3}}{{\lceil\frac{180\omega(n)}{41}\rceil}-1} < \frac{221}{180}.
\end{align*}
Using Property (\ref{p6}) we obtain 
\begin{align*}
    \frac{{\lceil \frac{180\omega(n)}{41}\rceil+\omega(n)-3}}{{\lceil \frac{180\omega(n)}{41}\rceil-1}}=\frac{5\omega(n)-3+\lceil\frac{16\omega(n)}{41}\rceil}{4\omega(n)-1+\lceil\frac{16\omega(n)}{41}\rceil}.
\end{align*}
\vspace{5mm}
We now see the behavior of $\frac{5\omega(n)-3+\lceil\frac{16\omega(n)}{41}\rceil}{4\omega(n)-1+\lceil\frac{16\omega(n)}{41}\rceil}$ based on the divisibility of $\omega(n)$ by $41$.\\ 
If $41\nmid \omega(n)$ then $\frac{\omega(n)}{41}\not\in \mathbb{Z^+}$ therefore using Property (\ref{p7}) and Property (\ref{p9}) we get

\begin{align*}
    \frac{5\omega(n)-3+\lceil\frac{16\omega(n)}{41}\rceil}{4\omega(n)-1+\lceil\frac{16\omega(n)}{41}\rceil}&=\frac{5\omega(n)-2+\frac{16\omega(n)}{41}-\{\frac{16\omega(n)}{41}\}}{4\omega(n)+\frac{16\omega(n)}{41}-\{\frac{16\omega(n)}{41}\}}\\
    &=\frac{221\omega(n)-82-41\{\frac{16\omega(n)}{41}\}}{180\omega(n)-41\{\frac{16\omega(n)}{41}\}}.
\end{align*}
Note that for any positive integer  $Q$, which is not divisible by $41$ can be written in the form $Q=41q+r$, where $q\in \mathbb{Z}_{\geq 0}$ and $r\in \{1,2,\dots, 40\}$. Therefore in particular for $Q=16\omega(n)$, we have $\{\frac{16\omega(n)} {41}\}\in \{\frac{1}{41},\frac{2}{41},\dots,\frac{40}{41}\}$ and so

\begin{align*}
    1\leq 41\{\frac{16\omega(n)}{41}\}\leq 40
\end{align*}
that is,
\begin{align}\label{3.3}
    221\omega(n)-122 \leq 221\omega(n)-82-41\{\frac{16\omega(n)}{41}\}\leq 221 \omega(n)-83
\end{align}
    and
    \begin{align}\label{3.4}
        180 \omega(n)-40 \leq {180\omega(n)-41\{\frac{16\omega(n)}{41}\}} \leq 180\omega(n)-1.
    \end{align}
Using (\ref{3.3}) and (\ref{3.4}) we finally have
\begin{align*}
     \frac{221\omega(n)-82-41\{\frac{16\omega(n)}{41}\}}{{180\omega(n)-41\{\frac{16\omega(n)}{41}\}}}\leq \frac{221\omega(n)-83}{180\omega(n)-40}.
\end{align*}
Now define $f: [1,\infty)	\rightarrow\mathbb{R}$ by $f(t)=\frac{221t-83}{180t-40}$. Then $f$ is a strictly increasing function of t in $[1,\infty)$ by Proposition \ref{propo 2}. Since $\lim_{t\rightarrow \infty} f(t)=\frac{221}{180}$ we have $f(t)<\frac{221}{180}$ for all $t\in [1,\infty)$. In particular for $t=\omega(n)$ we have
\begin{align*}
\frac{221\omega(n)-83}{180\omega(n)-40}<\frac{221}{180} 
\end{align*}
which immediately implies that 
\begin{align*}
   \frac{{\lceil\frac{180\omega(n)}{41}\rceil+\omega(n)-3}}{{\lceil\frac{180\omega(n)}{41}\rceil}-1} < \frac{221}{180}.
\end{align*}

Now if $41\mid \omega(n)$ then $\frac{\omega(n)}{41}\in \mathbb{Z^+}$. Therefore

\begin{align*}
     \frac{5\omega(n)-3+\lceil\frac{16\omega(n)}{41}\rceil}{4\omega(n)-1+\lceil\frac{16\omega(n)}{41}\rceil}=\frac{5\omega(n)-3+\frac{16\omega(n)}{41}}{4\omega(n)-1+\frac{16\omega(n)}{41}}=\frac{221\omega(n)-123}{180\omega(n)-41}.
\end{align*}
Now define $g: [1,\infty)	\rightarrow\mathbb{R}$ by $g(t)=\frac{221t-123}{180t-41}$. Then $g$ is a strictly increasing function of t in $[1,\infty)$ by Proposition \ref{propo 2}. Since $\lim_{t\rightarrow \infty} g(t)=\frac{221}{180}$ we have $g(t)<\frac{221}{180}$ for all $t\in [1,\infty)$. In particular for $t=\omega(n)$ we have
\begin{align*}
\frac{221\omega(n)-123}{180\omega(n)-41}<\frac{221}{180}
\end{align*}
which immediately implies that 
\begin{align*}
   \frac{{\lceil\frac{180\omega(n)}{41}\rceil+\omega(n)-3}}{{\lceil\frac{180\omega(n)}{41}\rceil}-1} < \frac{221}{180}.
\end{align*}

Therefore for any $\omega(n)\in \mathbb{Z^+}$ we have
\begin{align*}
   \frac{{\lceil\frac{180\omega(n)}{41}\rceil+\omega(n)-3}}{{\lceil\frac{180\omega(n)}{41}\rceil}-1} < \frac{221}{180}
\end{align*}
which shows that
\begin{align*}
I(n)<\frac{5}{4}\cdot\frac{7}{6}\frac{{\lceil\frac{180\omega(n)}{41}\rceil+\omega(n)-3}}{{\lceil\frac{180\omega(n)}{41}\rceil}-1}<\frac{1547}{864}<\frac{9}{5}.
\end{align*}
Therefore for $q_3\geq p_{\lceil\frac{180\omega(n)}{41}\rceil}$, $n$ can not be a friend of $10$. Hence, necessarily $q_3<p_{\lceil\frac{180\omega(n)}{41}\rceil}$.

\subsection{Proof of Theorem \ref{thm 1.3}}
Let $n=5^{2a_1}\cdot \prod_{2\leq i \leq \omega(n)}q^{2a_i}_i$ be a friend of 10. To prove this theorem it is enough to show that $q_4$ must be strictly less than $p_{\lceil\frac{390\omega(n)}{47}\rceil}$ by Theorem \ref{thm 3.4}. If possible, suppose that $q_4\geq p_{\lceil\frac{390\omega(n)}{47}\rceil}$. Then by Property (\ref{p4}) and Property (\ref{p5}) we have
\begin{align*}
    I(n)\leq I\left(5^{2a_1} \cdot 7^{2a_2} \cdot 11^{2a_3} \cdot \prod_{4\leq i\leq \omega(n)} p^{2a_i}_{\lceil \frac{390\omega(n)}{47}\rceil+i-4}\right)<\frac{5}{4}\cdot \frac{7}{6}\cdot\frac{11}{10} \cdot \prod_{4\leq i\leq \omega(n)}\frac{p_{\lceil \frac{390\omega(n)}{47}\rceil+i-4}}{p_{\lceil \frac{390\omega(n)}{47}\rceil+i-4}-1}.
\end{align*}

Using Remark \ref{r3} we get 
\begin{align*}
    I(n)&<\frac{5}{4}\cdot \frac{7}{6}\cdot\frac{11}{10}\cdot \prod_{4\leq i\leq \omega(n)} \frac{{\lceil \frac{390\omega(n)}{47}\rceil+i-4}}{{\lceil \frac{390\omega(n)}{47}\rceil+i-5}}\\
    &=\frac{5}{4}\cdot \frac{7}{6}\cdot\frac{11}{10}\cdot \frac{{\lceil \frac{390\omega(n)}{47}\rceil+\omega(n)-4}}{{\lceil \frac{390\omega(n)}{47}\rceil-1}}. 
\end{align*}

Now we will show that for any $\omega(n)\in\mathbb{Z^+}$ we have
\begin{align*}
    \frac{{\lceil \frac{390\omega(n)}{47}\rceil+\omega(n)-4}}{{\lceil \frac{390\omega(n)}{47}\rceil-1}}<\frac{437}{390}.
\end{align*}

Using Property (\ref{p6}) we obtain 
\begin{align*}
\frac{{\lceil \frac{390\omega(n)}{47}\rceil+\omega(n)-4}}{{\lceil \frac{390\omega(n)}{47}\rceil-1}}=\frac{9\omega(n)-4+\lceil\frac{14\omega(n)}{47}\rceil}{8\omega(n)-1+\lceil\frac{14\omega(n)}{47}\rceil}.    
\end{align*}
We now see the behavior of $\frac{9\omega(n)-4+\lceil\frac{14\omega(n)}{47}\rceil}{8\omega(n)-1+\lceil\frac{14\omega(n)}{47}\rceil} $ based on the divisibility of $\omega(n)$ by $47$.\\ 
If $47 \nmid \omega(n)$ then $\frac{\omega(n)}{47}\not\in \mathbb{Z^+}$ therefore using Property (\ref{p7}) and Property (\ref{p9}) we get
\begin{align*}
   \frac{9\omega(n)-4+\lceil\frac{14\omega(n)}{47}\rceil}{8\omega(n)-1+\lceil\frac{14\omega(n)}{47}\rceil}&=\frac{9\omega(n)-3+\frac{14\omega(n)}{47}-\{\frac{14\omega(n)}{47}\}}{8\omega(n)+\frac{14\omega(n)}{47}-\{\frac{14\omega(n)}{47}\}}\\
    &=\frac{437\omega(n)-141-47\{\frac{14\omega(n)}{47}\}}{390\omega(n)-47\{\frac{14\omega(n)}{47}\}}.
    \end{align*}
Note that for any positive integer  $Q$ which is not divisible by $47$ can be written in the form $Q=47q+r$, where $q\in \mathbb{Z}_{\geq 0}$ and $r\in \{1,2,\dots,46\}$. Therefore in particular for $Q=14\omega(n)$, we have $\{\frac{14\omega(n)} {47}\}\in \{\frac{1}{47},\frac{2}{47},\dots, \frac{46}{47}\}$ and so

\begin{align*}
    1\leq 47\{\frac{14\omega(n)}{47}\}\leq 46
\end{align*}
that is,
\begin{align}\label{3a}
    437\omega(n)-187 \leq 437\omega(n)-141-47\{\frac{14\omega(n)}{47}\}\leq 437 \omega(n)-142
\end{align}
    and
    \begin{align}\label{3b}
        390\omega(n)-46 \leq {390\omega(n)-47\{\frac{14\omega(n)}{47}\}} \leq 390\omega(n)-1.
    \end{align}

Using (\ref{3a}) and (\ref{3b}) we finally have
\begin{align*}
    \frac{437\omega(n)-141-47\{\frac{14\omega(n)}{47}\}}{{390\omega(n)-47\{\frac{14\omega(n)}{47}\}}}\leq \frac{437\omega(n)-142}{390\omega(n)-46}.
\end{align*}
Now define $f: [1,\infty)	\rightarrow\mathbb{R}$ by $f(t)=\frac{437t-142}{390t-46}$. Then $f$ is a strictly increasing function of t in $[1,\infty)$ by Proposition \ref{propo 2}. Since $\lim_{t\rightarrow \infty} f(t)=\frac{437}{390}$ we have $f(t)<\frac{437}{390}$ for all $t\in [1,\infty)$. In particular for $t=\omega(n)$ we have
\begin{align*}
\frac{437\omega(n)-142}{390\omega(n)-46}<\frac{437}{390} 
\end{align*}
which immediately implies that 
\begin{align*}
   \frac{{\lceil \frac{390\omega(n)}{47}\rceil+\omega(n)-4}}{{\lceil \frac{390\omega(n)}{47}\rceil-1}} < \frac{437}{390}.
\end{align*}

Now if $47\mid \omega(n)$ then $\frac{\omega(n)}{47}\in \mathbb{Z^+}$. Therefore

\begin{align*}
    \frac{9\omega(n)-4+\lceil\frac{14\omega(n)}{47}\rceil}{8\omega(n)-1+\lceil\frac{14\omega(n)}{47}\rceil}=\frac{9\omega(n)-4+\frac{14\omega(n)}{47}}{8\omega(n)-1+\frac{14\omega(n)}{47}}=\frac{437\omega(n)-188}{390\omega(n)-47}.
\end{align*}
Now define $g: [1,\infty)	\rightarrow\mathbb{R}$ by $g(t)=\frac{437t-188}{390t-47}$. Then $g$ is a strictly increasing function of t in $[1,\infty)$ by Proposition \ref{propo 2}. Since $\lim_{t\rightarrow \infty} g(t)=\frac{437}{390}$ we have $g(t)<\frac{437}{390}$ for all $t\in [1,\infty)$. In particular for $t=\omega(n)$ we have
\begin{align*}
\frac{437\omega(n)-188}{390\omega(n)-47}<\frac{437}{390} 
\end{align*}
which immediately implies that 
\begin{align*}
   \frac{{\lceil \frac{390\omega(n)}{47}\rceil+\omega(n)-4}}{{\lceil \frac{390\omega(n)}{47}\rceil-1}} < \frac{437}{390}.
\end{align*}
Therefore for any $\omega(n)\in \mathbb{Z^+}$ we have
\begin{align*}
   \frac{{\lceil \frac{390\omega(n)}{47}\rceil+\omega(n)-4}}{{\lceil \frac{390\omega(n)}{47}\rceil-1}} < \frac{437}{390}
\end{align*}
which shows that
\begin{align*}
    I(n)<\frac{5}{4}\cdot \frac{7}{6}\cdot\frac{11}{10}\frac{{\lceil \frac{390\omega(n)}{47}\rceil+\omega(n)-4}}{{\lceil \frac{390\omega(n)}{47}\rceil-1}}<\frac{33649}{18720}<\frac{9}{5}.
\end{align*}

Therefore for $q_4\geq p_{\lceil\frac{390\omega(n)}{47}\rceil}$, $n$ can not be a friend of $10$. Hence necessarily $q_4<p_{\lceil\frac{390\omega(n)}{47}\rceil}$.

\section{Conclusion}
We use an elementary approach to obtain bounds for second, third, and fourth smallest prime divisors of  friendly number of 10. However, we are unable to establish upper bounds for other prime divisors in terms of the number of distinct prime divisors, as we are unable to find positive integers $A, B$ with $A\geq B$ such that the methods used in Theorems \ref{thm 1.1}, \ref{thm 1.2}, and \ref{thm 1.3} can be applied to show that $q_5$ must be strictly less than $p_{\lceil\frac{A\omega(n)}{B}\rceil}$. $10$ is not the only number whose status is unknown in fact the status of $14$, $15$, $20$ and many others are active topics for Research. It may be proving whether $10$ has a friend or not is as much as difficult as finding an odd perfect number. However computer search shows that if $10$ has a friend then its smallest friend must be strictly greater than  $10^{30}$\cite{OEIS}.

\section{Acknowledgements}
We would like to thank the anonymous referee for their valuable comments and suggestions, which have greatly contributed to the improvement of this paper.

\end{document}